\documentclass{amsart}
\usepackage{graphicx, amsmath, amssymb, latexsym, cite}

\newtheorem{theorem}{Theorem}[section]
\newtheorem{lemma}[theorem]{Lemma}
\newtheorem{proposition}[theorem]{Proposition}
\newtheorem{corollary}[theorem]{Corollary}
\newtheorem{conjecture}[theorem]{Conjecture}

\theoremstyle{definition}
\newtheorem{definition}[theorem]{Definition}
\newtheorem{example}[theorem]{Example}

\theoremstyle{remark}

\numberwithin{equation}{section}

\newcommand{\sgn}{\operatorname{sgn}}

\newcommand{\R}{{\mathbb R}}

\newcommand{\Z}{{\mathbb Z}}

\newcommand{\N}{{\mathbb N}}
\newcommand{\T}{{\mathbb T}}
\newcommand{\bfk}{{\mathbf k}}
\newcommand{\lap}[1]{\sqrt{-\Delta_{#1}}}
\newcommand{\I}{\mathcal I}
\newcommand{\tgamma}{\tilde \gamma}
\newcommand{\diam}{\operatorname{diam}}

\renewcommand{\epsilon}{\varepsilon}

\newcommand{\supp}{{\operatorname{supp}}}

\begin{document}

\title[integrals of eigenfunctions over curves]{explicit bounds on integrals of eigenfunctions over curves in surfaces of nonpositive curvature}


\author{Emmett L. Wyman}
\address{Johns Hopkins University}
\curraddr{}
\email{ewyman3@math.jhu.edu}
\thanks{}


\date{}
\dedicatory{}

\begin{abstract}
Let $(M,g)$ be a compact Riemannian surface with nonpositive sectional curvature and let $\gamma$ be a closed geodesic in $M$. And let $e_\lambda$ be an $L^2$-normalized eigenfunction of the Laplace-Beltrami operator $\Delta_g$ with $-\Delta_g e_\lambda = \lambda^2 e_\lambda$. Sogge, Xi, and Zhang ~\cite{Gauss} showed using the Gauss-Bonnet theorem that
\[
	\int_\gamma e_\lambda \, ds = O((\log\lambda)^{-1/2}),
\]
an improvement over the general $O(1)$ bound. We show this integral enjoys the same decay for a wide variety of curves, where $M$ has \emph{nonpositive} sectional curvature. These are the curves $\gamma$ whose geodesic curvature avoids, pointwise, the geodesic curvature of circles of infinite radius tangent to $\gamma$.
\end{abstract}

\maketitle


\section{Introduction}

\subsection{Background}

Let $(M,g)$ be a compact, boundaryless, 2-dimensional Riemannian manifold. Let $\Delta_g$ denote the Laplace-Beltrami operator and $e_\lambda$ an $L^2$-normalized eigenfunction of $\Delta_g$ on $M$, i.e.
\[
    -\Delta_g e_\lambda = \lambda^2 e_\lambda \qquad \text{ and } \qquad \| e_\lambda \|_{L^2(M)} = 1.
\]
Good ~\cite{Good} and Hejhal ~\cite{Hej} showed that if $M$ is a hyperbolic surface and $\gamma$ is a periodic geodesic in $M$,
\begin{equation}\label{good hejhal bound}
	\int_\gamma e_\lambda \, ds = O(1).
\end{equation}
These geodesic period integrals, and more generally the Fourier series of eigenfunctions restricted to periodic geodesics, are of interest in the spectral theory of automorphic forms.

Good and Hejhal's bound was later extended to the general Riemannian setting by Zelditch ~\cite{ZelK} who provided the following powerful result. Let $M$ be any $n$-dimensional, compact, Riemannian manifold without boundary. Let $e_j$ for $j = 1,2,\ldots$ comprise a Hilbert basis of eigenfunctions with corresponding eigenvalues $\lambda_j$. If $\Sigma$ is a $d$-dimensional submanifold in $M$ and $d\sigma$ is a smooth, compactly supported multiple of the surface measure on $\Sigma$, then Zelditch provides a Kuznecov asymptotic formula,
\begin{equation} \label{zelditch kuz}
	\sum_{\lambda_j \leq \lambda} \left| \int_\Sigma e_j \, d\sigma \right|^2 \sim \lambda^{n-d} + O(\lambda^{n-d-1})
\end{equation}
where the implicit constant in front of the main term is nonzero provided $d\sigma$ is nonnegative and not identically zero.
As a consequence, we have
\[
	\sum_{\lambda_j \in [\lambda, \lambda+1]} \left| \int_\Sigma e_j \, d\sigma \right|^2 = O(\lambda^{n-d-1}).
\]
Taking the $n = 2$, $d = 1$ case provides Good and Hejhal's $O(1)$ bound\footnote{Zelditch's Kuznecov formula also tells us that generic eigenfunctions satisfy much better bounds. One can use an extraction argument to show that if $R(\lambda) \to \infty$, there exists a density $1$ subsequence of eigenfunctions for which the left hand side of \eqref{good hejhal bound} is $O(R(\lambda) \lambda^{-d/2})$, or $O(R(\lambda)\lambda^{-1/2}$ in the case of a curve.} regardless of hypotheses on the curvature of $M$ or $\gamma$.

In ~\cite{Rez}, Reznikov demonstrates the bound \eqref{good hejhal bound} for both periodic geodesics and circles in hyperbolic surfaces, and goes further to conjecture the following.
\begin{conjecture}[Reznikov] Let $M$ be a compact hyperbolic surface and $\gamma$ a closed geodesic or circle in $M$. Then,
\[
	\int_\gamma e_\lambda \, ds = O(\lambda^{-1/2 + \epsilon})
\]
for all $\epsilon > 0$.
\end{conjecture}
\noindent The first improvement\footnote{Reznikov's conjectured bound, or any polynomial improvement to the bound of \eqref{good hejhal bound} for that matter, seems to be inaccessible with standard techniques. But, it does hold for circles in the flat torus as we will see later in this section.} towards Reznikov's conjecture is due to Chen and Sogge ~\cite{CSper}, who obtained a little-$o$ improvement over the standard $O(1)$ bound for geodesics in compact surfaces with (not necessarily constant) negative curvature. Their strategy involved a lift to the universal cover, using the Hadamard parametrix to write the relevant quantity as an oscillatory integral with a geometric phase function, and using the Gauss-Bonnet theorem to show that critical points of the phase function are isolated. Recently Sogge, Xi, and Zhang ~\cite{Gauss} improved this bound further to $O((\log \lambda)^{-1/2})$ while also allowing the sectional curvature of $M$ to vanish of finite order.

Following Chen and Sogge's strategy, the author ~\cite{emmett1} used Jacobi fields to show the little-$o$ bound of ~\cite{CSper} holds in the broader scenario where $M$ has nonpositive sectional curvature and the geodesic curvature of $\gamma$ avoids that of circles of infinite radius tangent to $\gamma$. As a corollary, integrals of eigenfunctions over both geodesics and circles in hyperbolic manifolds enjoy the Chen and Sogge's $o(1)$ bound. Following Sogge, Xi, and Zhang's example, we improve this to $O((\log \lambda)^{-1/2})$ decay under the same hypotheses of ~\cite{emmett1}, albeit without the weakened sectional curvature hypotheses of ~\cite{Gauss}.

\subsection{Statement of results}

We require some notation to state our main result. If $\gamma$ is a curve in $M$, we denote by $\kappa_\gamma(t)$ the geodesic curvature of $\gamma$ at $t$, i.e.
\[
    \kappa_\gamma(t) = \frac{1}{|\gamma'(t)|} \left| \frac{D}{dt} \frac{\gamma'(t)}{|\gamma'(t)|} \right|,
\]
where $D/dt$ is the covariant derivative in the parameter $t$. For fixed $p \in M$ and $v \in T_pM$, we denote by $v^\perp$ a choice of vector in $T_pM$ for which $|v^\perp| = |v|$ and $\langle v^\perp, v \rangle = 0$. We denote by $S_pM$ the unit-length vectors in $T_pM$, and we let $SM$ denote the unit sphere bundle over $M$.

To state our main result, we need a function $\mathbf k$ on the unit sphere bundle of $M$ representing a ``critical geodesic curvature" which we assume $\gamma$ avoids to obtain decay of the integral in \eqref{good hejhal bound}.

\begin{definition} \label{def k}
    Fix $p \in M$ and $v \in S_pM$ and let $t \mapsto \zeta(t)$ the unit speed geodesic with $\zeta(0) = p$ and $\zeta'(0) = v$. Let $J$ be a Jacobi field along $\zeta$ satisfying
    \begin{equation} \label{J initial condition}
        J(0) = \zeta'(0)^\perp.
    \end{equation}
    We let $\bfk(v)$ denote the unique number such that
    \begin{equation} \label{J bounded}
        |J(r)| = O(1) \quad \text{ for } r \leq 0
    \end{equation}
    if $J$ satisfies the additional initial condition
    \begin{equation} \label{J' initial condition}
        \frac{D}{dr} J(0) = \bfk(v) J(0).
    \end{equation}
\end{definition}

$\bfk$ is a well-defined, continuous function on $SM$ by ~\cite[Proposition 4.1]{emmett1}, though we include the proof here as Proposition \ref{def k good} for completeness. The geometric meaning of $\bfk$ is clearer after a lift to the universal cover. By the theorem of Hadamard, we identify the universal cover of $M$ with $(\R^2, \tilde g)$, where $\tilde g$ is the pullback of the metric tensor $g$ through the covering map. If $p \in M$ and $\tilde p \in \R^2$ a lift of $p$ through the covering map, and $v \in S_pM$ and $\tilde v \in S_{\tilde p}\R^2$ the lift of $v$, the argument in Proposition \ref{def k good} reveals $\mathbf k(v)$ to be the limiting curvature of the circle at $\tilde p$ with center taken to infinity along the geodesic ray in direction $-\tilde v$ (see Figure \ref{kfig}). In the flat case, $\mathbf k \equiv 0$. If $M$ is a hyperbolic surface with sectional curvature $-1$, part (2) of Lemma \ref{large radius} tells us $\bfk \equiv 1$, the curvature of a horocycle in the hyperbolic plane. Our main result is as follows.

\begin{figure}
	\centering
	\includegraphics[width=.8\textwidth]{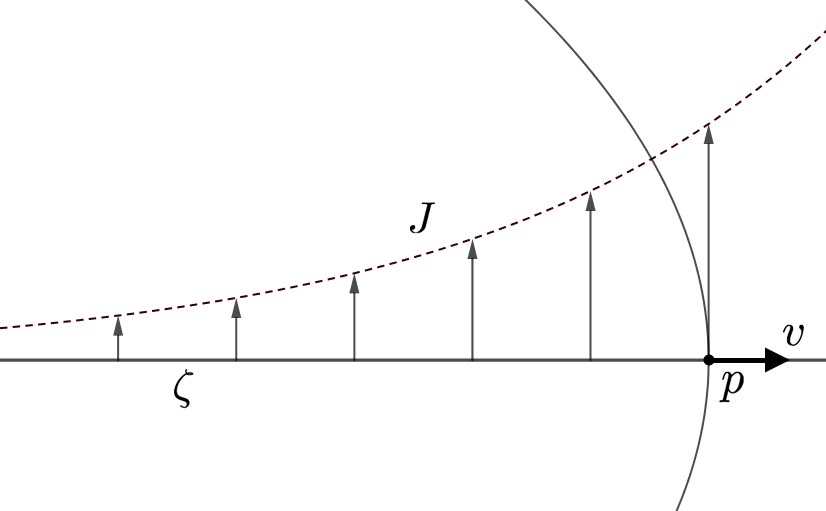}
	\caption{A Jacobi Field $J$ satisfying Definition \ref{def k}. $\bfk$ is equal to the geodesic curvature of the limiting circle on the right.}
	\label{kfig}
\end{figure}

\begin{theorem} \label{main theorem}
Let $(M,g)$ be a compact, boundaryless, 2-dimensional Riemannian manifold with nonpositive sectional curvature. Let $b$ be a smooth, compactly supported function on $\R$ and $\gamma$ be a unit-speed curve in $M$ parametrized on an interval containing the support of $b$. Then,
\begin{equation} \label{main theorem bound}
	\sum_{\lambda_j \in [\lambda, \lambda + (\log \lambda)^{-1}]} \left| \int b(s) e_j(\gamma(s)) \, ds \right|^2 = O((\log \lambda)^{-1})
\end{equation}
provided that
\begin{equation} \label{curvature hypotheses}
        \kappa_\gamma(s) \neq \bfk(\gamma'^\perp(s)) \quad \text{ and } \quad \kappa_\gamma(s) \neq \bfk(-\gamma'^\perp(s))
\end{equation}
for each $s \in \supp b$.
\end{theorem}

We obtain our desired bound as an immediate corollary to Theorem \ref{main theorem}.

\begin{theorem} \label{main corollary}
    Let $M$, $b$, and $\gamma$ be as in Theorem \ref{main theorem} with $\gamma$ satisfying \eqref{curvature hypotheses}. Then,
    \begin{equation} \label{main corollary bound}
        \int b(s) e_\lambda(\gamma(s)) \, ds = O((\log \lambda)^{-1/2}).
    \end{equation}
\end{theorem}

To prove Theorem \ref{main theorem}, we follow Sogge, Xi, and Zhang's argument in ~\cite{Gauss} to reduce the problem to an oscillatory integral with a geometric phase function. We then exploit properties of the curvature of circles of increasing radius to obtain bounds on the Hessian of the phase function, and conclude with a stationary phase argument.

There are a couple useful corollaries to Theorem \ref{main theorem} which are worth pointing out.
We will find by part (2) of Lemma \ref{large radius} that if the sectional curvature $K$ is bounded by
\[
	0 \geq K_0 \geq K \geq K_1
\]
for some constants $K_0$ and $K_1$, then
\[
	\sqrt{-K_0} \leq \bfk \leq \sqrt{-K_1}.
\]
This yields an easy criterion for determining $\gamma$ which satisfy the hypotheses \eqref{curvature hypotheses} of Theorem \ref{main theorem}.

\begin{corollary}\label{criterion corollary}
	Let $M$, $b$, and $\gamma$ be as in Theorem \ref{main theorem}, and suppose the sectional curvature $K$ of $M$ is bounded as above. If
	\[
		\kappa_\gamma(s) < \sqrt{-K_0} \qquad \text{ or } \qquad \kappa_\gamma(s) > \sqrt{-K_1}
	\]
	for all $s \in \supp b$, then \eqref{main theorem bound} and hence \eqref{main corollary bound} hold.
\end{corollary}

Note if the sectional curvature of $M$ is strictly negative, we can take $K_0 > 0$. The corollary then implies the main result in ~\cite{Gauss} for geodesics minus the weakened hypotheses which allow $K$ to vanish of finite order. If $\gamma$ is a geodesic circle, we can say something else.

\begin{corollary} \label{circle corollary} Let $M$ and $b$ be as in Theorem \ref{main theorem}, let the sectional curvature $K$ be bounded between $K_0$ and $K_1$ be as above, and let $\gamma$ be a unit-speed geodesic circle of radius $r > 0$. Then, \eqref{main theorem bound} and hence \eqref{main corollary bound} hold provided
\begin{align*}
	r &< \frac{1}{2\sqrt{-K_0}} \log\left( \frac{\sqrt{-K_1} + \sqrt{-K_0}}{\sqrt{-K_1} - \sqrt{-K_0}} \right)  &\text{ in the case } 0 > K_0 > K_1, \\
	r &< 1/\sqrt{-K_1}  &\text{ in the case } 0 = K_0 > K_1,
\end{align*}
or for all $r$ in the case where the sectional curvature is constant.
\end{corollary}

This corollary follows from the previous corollary, part (3) of Lemma \ref{large radius}, and the fact that the inverse of the hyperbolic cotangent function is
\[
	\operatorname{arccoth}(x) = \frac{1}{2} \log\left( \frac{x + 1}{x - 1} \right).
\]


\subsection{Examples} There are a couple model settings -- the round sphere and the flat torus -- which help to illustrate the necessity of the hypotheses of Theorem \ref{main theorem}. In what follows, $\gamma$ will always be a unit speed curve and $b$ will be a smooth, nonnegative, compactly supported function on $\R$.

\begin{example}[The Sphere]\label{sphere example}
We require our manifold have nonpositive sectional curvature in order to construct $\bfk$ in Definition \ref{def k}, but one may ask anyway if it is possible to impose some conditions on a curve in a positively curved manifold so that we obtain some decay like \eqref{main corollary bound}. We are able to give a negative answer to this question by considering the standard sphere $S^2$. The key ideas here are Zelditch's Kuznecov formula \eqref{zelditch kuz} and the fact that all of the eigenfunctions $\lambda$ are of the form
\[
	\lambda = \sqrt{k(k+1)} \qquad \text{ for some } k = 0,1,2,\ldots.
\]
(See ~\cite[Section 3.4]{Hang} or ~\cite[Theorem 3.1]{helgason}.)

Let $e_j$ be an orthonormal basis of eigenfunctions on $S^2$. For each distinct eigenvalue $\lambda$, construct a new eigenfunction $e_\lambda$ by
\[
	e_\lambda = \frac{\displaystyle \sum_{\lambda_j = \lambda} \left( \int b(s) e_j(\gamma(s)) \, ds \right) \overline{e_j} }{\displaystyle \left( \sum_{\lambda_j = \lambda} \left| \int b(s) e_j(\gamma(s)) \, ds \right|^2 \right)^{1/2}}.
\]
Note that $e_\lambda$ is $L^2$-normalized and that
\[
	\int b(s) e_\lambda(\gamma(s)) \, ds = \left( \sum_{\lambda_j = \lambda} \left| \int b(s) e_j(\gamma(s)) \, ds \right|^2 \right)^{1/2}.
\]
By \eqref{zelditch kuz} and the assumption that $b$ is nonnegative, there exists some large constant $c$ so that
\[
	\sum_{\lambda_j \in [\lambda,\lambda+c]} \left| \int b(s) e_j(\gamma(s)) \, ds \right|^2 \geq 1.
\]
However, there are at most $c$ distinct eigenvalues in the interval $[\lambda, \lambda + c]$. Hence, every interval of length $c$ has an eigenvalue $\lambda$ for which
\[
	\int b(s) e_\lambda(\gamma(s)) \, ds \geq 1/c.
\]
Hence, there is no version of Theorem \ref{main corollary} which will hold on the sphere.
\end{example}

\begin{example}[The Torus] \label{torus example} Let $\T^2 = \R^2/2\pi \Z^2$ denote the flat torus. As noted before, $\bfk \equiv 0$ and every curve $\gamma$ with nonvanishing geodesic curvature satisfies the hypotheses \eqref{curvature hypotheses} of Theorem \ref{main theorem}. Suppose $\gamma$ is one such curve parametrized by arc-length. Let $e_\lambda$ be any $L^2$-normalized eigenfunction on $\T^2$. Note
\[
	e_\lambda(x) = \sum_{\substack{m \in \Z^2 \\ |m| = \lambda}} a_m e^{im\cdot x} 
\]
for some coefficients $a_m$ satisfying
\[
	\sum_{\substack{m \in \Z^2 \\ |m| = \lambda}} |a_m|^2 = \frac{1}{(2\pi)^2}.
\]
By Cauchy-Schwartz,
\begin{align*}
	\left| \int b(s) e_\lambda(\gamma(s)) \, ds \right| &= \left| \sum_{\substack{m \in \Z^2 \\ |m| = \lambda}} a_m \int b(s) e^{im\cdot \gamma(s)} \, ds \right| \\
	&\leq \frac{1}{2\pi} \#\{ m \in \Z^2 : |m| = \lambda\}^{1/2} \sup_{|m| = \lambda} \left| \int b(s) e^{im\cdot \gamma(s)} \, ds \right|.
\end{align*}
Divisor bounds for Gaussian integers yields an essentially sharp estimate
\[
	\#\{ m \in \Z^2 : |m| = \lambda \} = O(\lambda^\epsilon)
\]
for all $\epsilon > 0$, while a standard stationary phase argument yields a sharp uniform bound of $O(\lambda^{-1/2})$ on the integral. Hence,
\[
	\int b(s) e_\lambda(\gamma(s)) \, ds = O(\lambda^{-1/2 + \epsilon}),
\]
which is much better than the bound \eqref{main corollary bound} in Theorem \ref{main corollary}. This example can be seen to be sharp by taking $\gamma$ to be a circle, $b \equiv 1$, and $a_m$ constant over $|m| = \lambda$.

It is worth remarking that, again using standard stationary phase arguments, we obtain a polynomial improvement over the bound \eqref{main corollary bound} even if the geodesic curvature of $\gamma$ were to vanish of finite order. However, problems occur when $\gamma$ contains a line segment. If $\gamma$ \emph{is} a line segment in $\T^2$, we may construct a sequence of exponentials which are essentially constant on $\gamma$. We select a sequence $m_k$ on the integer lattice with $|m_k| \to \infty$ whose distance $|m_k \cdot  \gamma'(s)|$ from the space of normal vectors to $\gamma$ in $\R^2$ vanishes in the limit. Then,
\[
	\left| \int b(s) e^{im_k \cdot \gamma(s)} \, ds \right| \longrightarrow \left| \int b(s) \, ds \right|,
\]
and the conclusion of Theorem \ref{main corollary} does not hold.
\end{example}

In the example for the torus above, we demonstrate the bound
\begin{equation} \label{final intro int}
	\int b(s) e_\lambda(\gamma(s)) \, ds = O(1)
\end{equation}
cannot be improved if $\gamma$ is a geodesic segment. The analogous situation on a compact hyperbolic surface is when $\gamma$ is a segment of a horocycle. It is natural to ask if this bound is still sharp. The answer is: probably not. Assuming the quantum unique ergodicity conjecture, any sequence of eigenfunctions on a compact hyperbolic surface is quantum ergodic. Recently, Canzani, Galkowski, and Toth ~\cite{CGT} showed that if $e_\lambda$ is a quantum ergodic sequence of eigenfunctions, the integral in \eqref{final intro int} is necessarily $o(1)$ for all smooth curves $\gamma$. On the other hand, if one can construct a sequence of eigenfunctions saturating \eqref{final intro int} for some horocycle $\gamma$, one will have provided a negative answer to the quantum unique ergodicity conjecture.

\vspace{0.6em}
\noindent \textbf{Acknowledgements.} The author would like to thank his advisor, Christopher Sogge, for providing the problem, relevant materials, and support. Thanks as well to Yakun Xi and Cheng Zhang, whose work this paper models. This work is supported in part by NSF grant DMS-1069175.


\section{Standard reduction and lift to the universal cover} \label{STANDARD REDUCTION}

We employ a standard strategy to reduce the bound in Theorem \ref{main theorem} to one involving a sum over the deck transformations in the universal cover of oscillatory integrals with some geometric phase function. Our presentation of this reduction is only superficially different from ~\cite{CSper} and ~\cite{Gauss}. The idea to lift the problem to the universal cover originally appeared in B\'erard's celebrated paper ~\cite{Berard} in which he obtained a log improvement to the bound on the remainder term of the sharp Weyl law for compact manifolds with nonpositive sectional curvature. Other elements of this reduction, including uniformizing the sum in Theorem \ref{main theorem} over a Schwartz-class function and writing the result as some integral against the kernel of the half-wave operator, are very standard and appear in many results in global harmonic analysis (see for example ~\cite{BGT,SZDuke,Berard,CSper,Gauss} and ~\cite{SFIO,Hang} for further reading).

Let $e_j$ for $j = 0,1,2, \ldots$ be a Hilbert basis of eigenfunctions with corresponding eigenvalues $\lambda_j$. To prove Theorem \ref{main theorem}, it suffices to show
\begin{equation} \label{reduction 1}
	\sum_j \chi(T(\lambda_j - \lambda)) \left| \int b(s) e_j(\gamma(s)) \, ds \right|^2 \lesssim T^{-1} + e^{CT} \lambda^{-1/2}
\end{equation}
where $\chi$ is a nonnegative, Schwartz-class function on $\R$ with $\chi(0) = 1$ and $\supp \hat \chi \subset [-1,1]$, and where
\begin{equation} \label{T = clog}
	T = c \log \lambda
\end{equation}
for some appropriately small positive constant $c$. Using a partition of unity, we write $b$ as a sum
\[
	b = \sum_k b_k
\]
of smooth, compactly supported functions with small support. By Cauchy-Schwarz, the left hand side of \eqref{reduction 1} is bounded by the number of $b_k$ in the sum times
\[
	\sum_k \sum_j \chi(T(\lambda_j - \lambda)) \left| \int b_k(s) e_j(\gamma(s)) \, ds \right|^2.
\]
Hence, it suffices to prove \eqref{reduction 1} where $b$ has arbitrarily small support. After expanding the integral in \eqref{reduction 1}, the left hand side of \eqref{reduction 1} is
\begin{align*}
	&\sum_j \chi(T(\lambda_j - \lambda)) \iint b(s) b(t) e_j(\gamma(s)) \overline{e_j(\gamma(t))} \, ds \, dt\\
	&= \frac{1}{2\pi} \sum_j \iiint b(s) b(t) \hat \chi(\tau) e^{i\tau T(\lambda_j - \lambda)} e_j(\gamma(s)) \overline{e_j(\gamma(t))} \, ds \, dt \, d\tau\\
	&= \frac{1}{2\pi T} \sum_j \iiint b(s) b(t) \hat \chi(\tau/T) e^{i\tau (\lambda_j - \lambda)} e_j(\gamma(s)) \overline{e_j(\gamma(t))} \, ds \, dt \, d\tau\\
	&= \frac{1}{2\pi T} \iiint b(s) b(t) \hat \chi(\tau/T) e^{-i\tau \lambda} e^{i\tau \lap g}(\gamma(s), \gamma(t)) \, d\tau \, ds \, dt
\end{align*}
where the second line follows from the Fourier inversion formula, the third by a change of variables, and the fourth by writing out the kernel
\[
	e^{i\tau \lap g}(x,y) = \sum_j e^{i\tau \lambda_j} e_j(x) \overline{e_j(y)}
\]
of the half-wave operator $e^{it\lap g}$. Hence, \eqref{reduction 1} will follow from
\[
    \left| \iiint b(s)b(t) \hat \chi(\tau/T) e^{-i\lambda \tau} e^{i\tau \lap g}(\gamma(s), \gamma(t)) \, d\tau \, ds \, dt \right| \lesssim 1 + e^{CT} \lambda^{-1/2}.
\]
To simplify things, we scale the metric so that the injectivity radius of $M$ is at least $10$. Moreover, after perhaps restricting the support of $b$ using a partition of unity, we ensure that
\begin{equation} \label{support of b}
    \sup_{s,t \in \supp b} d_g(\gamma(s), \gamma(t)) \leq 1
\end{equation}
where $d_g$ is the distance function with respect to the metric $g$. Now we let $\beta \in C_0^\infty(\R)$ be a smooth cutoff function so that $\beta(\tau) = 1$ for $|\tau| \leq 2$ and $\beta(\tau) = 0$ for $|\tau| \geq 4$. We begin by proving the bound
\begin{equation} \label{local bound}
    \left| \iiint b(s)b(t) \beta(\tau) \hat \chi(\tau/T) e^{-i\lambda \tau} e^{i\tau \lap g}(\gamma(s), \gamma(t)) \, ds \, dt \, d\tau \right| \lesssim 1.
\end{equation}
As argued in ~\cite{CSper} and ~\cite{Gauss}, we have by the proof of Lemma 5.1.3 in ~\cite{SFIO}
\[
    \int \beta(\tau) \hat \chi(\tau/T) e^{-i\lambda \tau} e^{i\tau \lap g} (x,y) \, d\tau = \lambda^{1/2} \sum_\pm a_\pm(\lambda; d_g(x, y)) e^{\pm i \lambda d_g(x,y)} + O(1)
\]
where the amplitude $a_\pm$ satisfies bounds
\begin{equation} \label{local a bound 1}
    \left| \frac{d^j}{dr^j} a_\pm(\lambda; r) \right| \leq C_j r^{-j-1/2} \qquad \text{ if } r \geq \lambda^{-1}
\end{equation}
and
\begin{equation} \label{local a bound 2}
    |a_\pm(\lambda; r)| \leq C\lambda^{1/2} \qquad \text{ if } r \leq \lambda^{-1}.
\end{equation}
Hence \eqref{local bound} would follow from
\begin{equation} \label{local bound 2}
    \sum_\pm \left| \iint b(s)b(t) a_\pm(\lambda; d_g(\gamma(s), \gamma(t))) e^{\pm i \lambda d_g(\gamma(s), \gamma(t))} \, ds \, dt \right| \lesssim \lambda^{-1/2}.
\end{equation}
We let
\[
    r(s,t) = \sgn(s - t) d_g(\gamma(s), \gamma(t))
\]
and if necessary restrict the support of $b$ so that the map
\[
    (s,t) \mapsto (s,r(s,t))
\]
is a diffeomorphism on $\supp b \times \supp b$. By a change of coordinates, we write the integral in \eqref{local bound 2} as
\begin{align*} 
    \iint \tilde b(s,r) a_\pm(\lambda; |r|) &e^{\pm i \lambda |r|} \, ds \, dr\\
    &= \iint_{|r| \leq \lambda^{-1}} + \iint_{|r| \geq \lambda^{-1}} \tilde b(s,r) a_\pm(\lambda; |r|) e^{\pm i \lambda |r|} \, ds \, dr. 
\end{align*}
where $\tilde b(s,r) \, ds \, dr = b(s)b(t) \, ds \, dt$ where $r = r(s,t)$. The $|r| \leq \lambda^{-1}$ part is $O(\lambda^{-1/2})$ by \eqref{local a bound 2}. The same bound holds for the $|r| \geq \lambda^{-1}$ part after integrating by parts in $r$ once and applying \eqref{local a bound 1}. Hence, we have \eqref{local bound}.

What remains is to prove
\begin{align} 
    \nonumber \left| \iiint b(s)b(t)(1 - \beta(\tau)) \hat \chi(\tau/T) e^{-i\lambda \tau} e^{i\tau \lap g}(\gamma(s), \gamma(t)) \, ds \, dt \, d\tau \right|&\\
    \label{global bound} &\hspace{-8em}\lesssim 1 + e^{CT} \lambda^{-1/2}.
\end{align}
As in ~\cite{CSper} and ~\cite{Gauss}, we want to lift the kernel $e^{i \tau \lap g}(x,y)$ to the universal cover of $M$, but we also want to make use of Huygen's principle afterwards. So before we lift, we will rephrase the bound above using the kernel $\cos(\tau \lap g)$ instead of $e^{i \tau \lap g}$. By Euler's formula,
\[
    e^{i \tau \lap g} = 2 \cos(\tau \lap g) - e^{-i\tau \lap g},
\]
the contribution of the second term on the right side to \eqref{global bound} is
\begin{align*}
    \iiint &b(s) b(t) (1 - \beta(\tau)) \hat \chi(\tau/T) e^{-i\lambda \tau} e^{-i \tau \lap g}(\gamma(s),\gamma(t)) \, d\tau \, ds \, dt\\
    &= \sum_j \iiint (1 - \beta(\tau)) \hat \chi(\tau/T) e^{-i(\lambda_j + \lambda) \tau} e_j(\gamma(s)) \overline{e_j(\gamma(t))} \, d\tau \, ds \, dt \\
    &= 2\pi \sum_j \Phi_T(\lambda_j + \lambda) \iint b(s) b(t) e_j(\gamma(s)) \overline{e_j(\gamma(t))} \, ds \, dt
\end{align*}
where $\hat \Phi_T(\tau) = (1 - \beta(\tau)) \hat \chi(\tau/T)$. By integration by parts,
\[
    \Phi_T(\lambda) = O(\lambda^{-N}) \qquad \text{ for all } N \in \N.
\]
Hence,
\begin{align*}
    2\pi &\left| \sum_j \Phi_T(\lambda_j + \lambda) \iint b(s) b(t) e_j(\gamma(s)) \overline{e_j(\gamma(t))} \, ds \, dt \right|\\
    &\leq 2\pi \sum_j |\Phi_T(\lambda_j + \lambda)|\left| \int b(s) e_j(\gamma(s)) \, ds \right|^2\\
    &= O(\lambda^{-N}),
\end{align*}
where in the last line we invoke the standard $O(1)$ bound on the integral.
Hence, it suffices to show
\begin{align} \label{global bound 2}
    \nonumber \left| \iiint b(s)b(t) (1 - \beta(\tau)) \hat \chi(\tau/T) e^{-i\lambda \tau} \cos(\tau \lap g)(\gamma(s), \gamma(t)) \, d\tau \, ds \, dt  \right|& \\
&\hspace{-4em} \leq 1 + e^{CT} \lambda^{-1/2}.
\end{align}
Now we are ready to lift to the universal cover. Since $M$ has nonpositive sectional curvature, we identify its universal cover with $(\R^2, \tilde g)$ where $\tilde g$ is the pullback of the metric tensor $g$ through the covering map. Fix $\tilde f \in C_0^\infty(\R^2)$ and define $f \in C^\infty(M)$ by
\[
    f(x) = \sum_{\alpha \in \Gamma} \tilde f \circ \alpha(\tilde x)
\]
where $\Gamma$ is the set of deck transformation associated with the covering map and $\tilde x$ is a lift of $x$ to $\R^2$. Now if $\tilde u(t,x)$ satisfies the wave equation
\[
    \partial_t^2 \tilde u - \Delta_{\tilde g} \tilde u = 0
\]
with initial data 
\begin{align*}
    \tilde u(0,\tilde x) &= \tilde f(\tilde x) \quad \text{ and }\\
    \partial_t \tilde u(0,\tilde x) &= 0,
\end{align*}
then
\[
    u(t,x) = \sum_{\alpha \in \Gamma} \tilde u(t, \alpha(\tilde x))
\]
satisfies the wave equation in $M$ with respect to the metric $g$ with initial data
\begin{align*}
    u(0,x) &= f(x) \\
    \partial_t u(0,x) &= 0.
\end{align*}
Hence,
\[
    \cos(\tau \lap{g}) f(x) = \sum_{\alpha \in \Gamma} \cos(\tau \lap{\tilde g}) \tilde f(\alpha(\tilde x)),
\]
from which we obtain the relationship
\[
    \cos(\tau \lap g)(x,y) = \sum_{\alpha \in \Gamma} \cos(\tau \lap{\tilde g})(\alpha(\tilde x),\tilde y)
\]
between the kernels. Using this identity, we will have the bound \eqref{global bound 2} if we can show
\begin{equation}\label{lifted bound 1}
    \sum_{\alpha \in \Gamma} \left| \iint b(s)b(t) K_{T,\lambda}(\tilde \gamma_\alpha(s), \tilde \gamma(t)) \, ds \, dt \right| \lesssim 1 + e^{CT} \lambda^{-1/2}
\end{equation}
where $\tilde \gamma$ is a lift of $\gamma$ to the universal cover, $\tgamma_\alpha$ is shorthand for $\alpha \circ \tgamma$, and
\[
    K_{T,\lambda}(x, y) = \int (1 - \beta(\tau)) \hat \chi(\tau/T) e^{-i\lambda \tau} \cos(\tau \lap{\tilde g})(x,y) \, d\tau.
\]
Proposition 5.1\footnote{Sogge, Xi, and Zhang \cite{Gauss} prove this proposition using the Hadamard parametrix as it appears in B\'erard ~\cite{Berard}.} in ~\cite{Gauss} stated below, gives a crucial characterization of the kernel $K_{T,\lambda}$.

\begin{proposition}[Sogge, Xi, Zhang] \label{kernel bounds}
    If $d_{\tilde g}(x,y) \geq 1$ and $\lambda \gg 1$ we can write
    \[
        K_{T,\lambda}(x,y) = \lambda^{1/2} \sum_\pm a_\pm(T,\lambda; x,y) e^{\pm i \lambda d_{\tilde g}(x,y)} + R_{T,\lambda}(x,y),
    \]
    where
    \begin{equation} \label{KB a bounded}
        |a_\pm(T,\lambda;x,y)| \leq C,
    \end{equation}
    and if $N = 1,2,\ldots$ is fixed
        \begin{equation} \label{KB a bound x}
        \Delta_x^N a_\pm(T,\lambda; x,y) = O(e^{C_N d_{\tilde g}(x,y)})
    \end{equation}
    and
    \begin{equation} \label{KB a bound y}
        \Delta_y^N a_\pm(T,\lambda; x,y) = O(e^{C_N d_{\tilde g}(x,y)}),
    \end{equation}
    where $\Delta_x$ and $\Delta_y$ denote the operator $\Delta_{\tilde g}$ acting in the $x$ and $y$ variables, respectively. Additionally if the constant $c > 0$ in \eqref{T = clog} is small enough, 
    \begin{equation} \label{KB remainder bound}
        R_{T,\lambda}(x,y) = O(\lambda^{-1}) \qquad \text{ if } d_{\tilde g}(x,y) \geq 1
    \end{equation}
    and
    \begin{equation} \label{KB local bound}
        K_{T,\lambda}(x,y) = O(\lambda^{-1}) \qquad \text{ if } d_{\tilde g}(x,y) \leq 1.
    \end{equation}
\end{proposition}

By \eqref{KB local bound} , the contribution of the identity term to the sum in \eqref{lifted bound 1} is $O(\lambda^{-1})$, better than we need. Hence we need only check that
\begin{equation} \label{lifted bound 2}
    \sum_{\alpha \in \Gamma \setminus I} \left| \iint b(s)b(t) K_{T,\lambda}(\tilde \gamma_\alpha(s), \tilde \gamma(t)) \, ds \, dt \right| \lesssim 1 + e^{CT} \lambda^{-1/2}.
\end{equation}
where $I$ denotes the identity element in $\Gamma$. By Huygen's principle and since $\supp \hat \chi \subset [-1,1]$, the kernel $K_{T,\lambda}$ is supported on $d_{\tilde g}(x,y) \leq T$, and so the sum in \eqref{lifted bound 2} is finite. In fact, by volume comparison with the plane of constant curvature, there are $O(e^{CT})$ many terms in the sum. Since the injectivity radius of $M$ is at least $10$, \eqref{support of b} implies $d_{\tilde g}(\tilde \gamma_\alpha(s), \tilde \gamma(t)) \geq 8$ for $s,t \in \supp b$ if $\alpha \neq I$. Hence by \eqref{KB remainder bound},
\[
    \sum_{\alpha \neq I} \left| \iint b(s)b(t) R_{T,\lambda}(\tilde \gamma_\alpha(s), \tilde \gamma(t)) \, ds \, dt \right| \lesssim \lambda^{-1} e^{CT},
\]
again satisfying better bounds than required. Hence, it suffices to show
\begin{equation} \label{lifted bound 3}
    \lambda^{1/2} \sum_{\alpha \neq I} \left| \iint b(s)b(t) a_\pm(T,\lambda; \tilde \gamma_\alpha(s), \tilde \gamma(t)) e^{\pm i \lambda \phi_\alpha(s,t)} \, ds \, dt \right| \lesssim 1 + e^{CT} \lambda^{-1/2}
\end{equation}
where $\phi_\alpha(s,t) = d_{\tilde g}(\tilde \gamma_\alpha(s), \tilde \gamma(t))$.

We approach the bound \eqref{lifted bound 3} by splitting the sum into two parts. After perhaps smoothly extending $\gamma$ past its endpoints, the continuity of $\bfk$ allows us to select an open interval $\mathcal I$ containing $\supp b$ on which the curvature hypotheses \eqref{curvature hypotheses} are satisfied. We fix a constant $R$ independent of $T$ and $\lambda$ to be determined later, and let
\begin{equation} \label{def A}
    A = \{\alpha \in \Gamma : \inf_{s,t \in \mathcal I} d_{\tilde g}(\tilde \gamma_\alpha(s), \tilde \gamma(t)) \leq R \}.
\end{equation}
We will have the bound \eqref{lifted bound 3} if we can prove the following respective medium- and large- time bounds.

\begin{proposition}\label{medium time}
For any $\alpha \in \Gamma \setminus I$, there exists a constant $C_\alpha$ such that
\[
    \left| \iint b(s)b(t) a_\pm(T, \lambda; \tilde \gamma_\alpha(s), \tilde \gamma(t)) e^{\pm i \lambda \phi_\alpha(s,t)} \, ds \, dt \right| \leq C_\alpha \lambda^{-1/2}
\]
for $\lambda > 1$.
\end{proposition}

\begin{proposition} \label{large time}
There exists a constant $C$ independent of $T$ and $\lambda$ such that for every $\alpha \in \Gamma \setminus A$ with $\inf_{s,t \in \mathcal I} \phi_\alpha(s,t) \leq T$,
\[
    \left| \iint b(s)b(t) a_\pm(T, \lambda; \tilde \gamma_\alpha(s), \tilde \gamma(t)) e^{\pm i \lambda \phi_\alpha(s,t)} \, ds \, dt \right| \leq C e^{CT} \lambda^{-1}
\]
\end{proposition}

Since $A$ has a fixed number of terms, Proposition \ref{medium time} implies that the contribution of the set $A$ to the sum to \eqref{lifted bound 3} is $O(1)$. Moreover because $K_{T,\lambda}(x,y)$ vanishes for $d_{\tilde g}(x,y) > T$, we need only consider the terms $\alpha \in \Gamma \setminus A$ such that $\phi_\alpha(s,t) \leq T$ for some $s,t \in \supp b$. As noted earlier, there are $O(e^{CT})$ many such terms, and so Proposition \ref{large time} tells us the contribution of the set $\Gamma \setminus A$ to the sum in \eqref{lifted bound 3} is $O(e^{CT} \lambda^{-1/2})$, as desired. To prove Propositions \ref{medium time} and \ref{large time}, we will use the geometric tools from ~\cite{emmett1} to obtain bounds on the derivatives of $\phi_\alpha$ and then apply stationary phase.


\section{Phase function bounds and proof of Proposition \ref{medium time} } \label{PHASE FUNCTION BOUNDS}

To prove Propositions \ref{medium time} and \ref{large time}, we will need bounds on the derivatives of the phase function $\phi_\alpha$ for $\alpha \neq I$. First, we bound the mixed partial derivative $\partial_s \partial_t \phi_\alpha$, and second compute $\partial_s^2 \phi_\alpha$ and $\partial_t^2 \phi_\alpha$ in terms of the curvature $\kappa_\gamma$ of $\gamma$ and the curvature of circles. Following this, we prove Proposition \ref{medium time}. ~\cite{doCarmo} is the principal reference for the geometric arguments which follow.

Let $F : \I \times \I \times \R \to \R^2$ be the smooth map to the universal cover defined so that $r \mapsto F(s,t,r)$ is the constant-speed geodesic with $F(s,t,0) = \tgamma(t)$ and $F(s,t,1) = \tgamma_\alpha(s)$. If $\partial_s$, $\partial_t$, and $\partial_r$ are the coordinate vector fields living in the domain of $F$, then the Lie brackets $[\partial_s, \partial_t]$, $[\partial_s,\partial_r]$ and $[\partial_t,\partial_r]$ all vanish. Hence,
\begin{equation}\label{st commute}
    \frac{D}{ds} \partial_t F - \frac{D}{dt} \partial_s F = [\partial_s F, \partial_t F] = [F_* \partial_s , F_* \partial_t] = F_*[\partial_s, \partial_t] = 0
\end{equation}
and
\begin{equation} \label{r commute}
    \frac{D}{ds} \partial_r F - \frac{D}{dr} \partial_s F = 0 \quad \text{ and } \quad \frac{D}{dt} \partial_r F - \frac{D}{dr} \partial_t F = 0
\end{equation}
similarly. Now,
\[
    \phi_\alpha^2(s,t) = \int_0^1 |\partial_r F(s,t,r)|^2 \, dr,
\]
and taking a derivative in $s$ yields
\begin{align}
    \nonumber \phi_\alpha(s,t) \partial_s \phi_\alpha(s,t) &= \int_0^1 \left\langle \frac{D}{ds} \partial_r F(s,t,r), \partial_r F(s,t,r) \right\rangle \, dr\\
    \nonumber &= \int_0^1 \partial_r \langle \partial_s F(s,t,r), \partial_r F(s,t,r) \rangle \, dr\\
    \label{first derivative} &= \langle \tgamma_\alpha'(s), \partial_r F(s,t,1) \rangle
\end{align}
where the second line follows from \eqref{r commute} and the geodesic equation $\frac{D}{dr} \partial_r F = 0$, and the third line from the fundamental theorem of calculus and the observation that $\partial_s F(s,t,0) = 0$. Moreover, since the curves $\tgamma$ and $\tgamma_\alpha$ are disjoint, $\phi_\alpha$ is nonvanishing. We then have the following fact (also noted in ~\cite{CSper} and ~\cite{Gauss}): $\partial_s \phi_\alpha(s,t)$ vanishes if and only if $\tgamma_\alpha$ is perpendicular to the geodesic adjoining $\tgamma_\alpha(s)$ and $\tgamma(t)$, and similarly if $\partial_t \phi_\alpha$ vanishes. The gradient $\nabla \phi_\alpha(s,t)$ vanishes if and only if $\tgamma$ and $\tgamma_\alpha$ are both perpendicular to the geodesic adjoining $\tgamma_\alpha(s)$ and $\tgamma(t)$.

Now we compute the mixed partial derivative $\partial_s \partial_t \phi_\alpha$. Taking a derivative in $t$ of \eqref{first derivative} yields
\begin{equation} \label{mixed derivative computation}
    \partial_s \phi_\alpha(s,t) \partial_t \phi_\alpha(s,t) + \phi_\alpha(s,t) \partial_t \partial_s \phi_\alpha(s,t) = \left \langle \tgamma_\alpha'(s), \frac{D}{dr} \partial_t F(s,t,1) \right \rangle.
\end{equation}
We claim that
\begin{equation} \label{mixed bound claim}
    \left| \frac{D}{dr} \partial_s F(s,t,1) \right| \leq 1,
\end{equation}
which, along with the observation that $|\tgamma_\alpha'|$, $|\partial_s \phi_\alpha|$, and $|\partial_t \phi_\alpha |$ are all bounded above by $1$, yields the following bound.

\begin{lemma} \label{mixed bound}
\[
    |\partial_s \partial_t \phi_\alpha | \leq 2 \phi_\alpha^{-1}.
\]
\end{lemma}

\begin{proof}
In light of \eqref{mixed derivative computation}, it suffices only to verify our claim \eqref{mixed bound claim}. For fixed $s$ and $t$ we write $\partial_s F$ as the sum of parallel and perpendicular parts
\[  
	\partial_s F = f(r) \partial_r F + h(r) \partial_r F^\perp.
\]
Now,
\[
    1 = |\partial_s F(s,t,0)|^2 = f(0)^2 + h(0)^2,
\]
and since $\frac{D}{dr} \partial_r F = 0$ (and indeed also $\frac{D}{dr} \partial_r F^\perp = 0$), it suffices to show that
\begin{equation} \label{bounds on f and h}
    f'(1) = -f(0) \qquad \text{ and } \qquad |h'(1)| \leq |h(0)|.
\end{equation}
Since $f \partial_r F$ and $h \partial_r F^\perp$ are parallel and perpendicular Jacobi fields along $r \mapsto F(s,t,r)$, respectively, we have
\[
    f'' = 0,
\]
from which the first part of \eqref{bounds on f and h} follows, and
\[
    h'' + K h = 0
\]
where $\tilde K = \tilde K(F(s,t,r))$ is the sectional curvature of $(\R^2,\tilde g)$. We assume without loss of generality by our choice of $F^\perp$ that $h(0) \geq 0$. Since $\tilde K \leq 0$, $h$ vanishes uniquely at $1$, and so $h(r) > 0$ for $0 < r < 1$. Again since $\tilde K \leq 0$, $h$ is convex on $[0,1]$. In particular,
\[
    0 \leq h(r) \leq h(0)(1 - r),
\]
so
\[
    0 \geq h'(1) \geq -h(0),
\]
from which we obtain the second part of \eqref{bounds on f and h}.
\end{proof}

Now we compute $\partial_s^2 \phi_\alpha$. Fix $t_0$ and let $r \mapsto \zeta(s,r)$ denote the unit speed geodesic with $\zeta(s,0) = \tgamma(t_0)$ and $\zeta(s,\phi_\alpha(s,t_0)) = \tgamma_\alpha(s)$. To avoid ambiguity in the notation, we fix $s_0$ and let $r_0 = \phi_\alpha(s_0,t_0)$, and compute $\partial_s^2 \phi_\alpha(s_0,t_0)$. By \eqref{first derivative}, 
\[
    \partial_s \phi_\alpha(s,t_0) = \left \langle \tgamma_\alpha'(s), \partial_r \zeta(s,\phi_\alpha(s,t_0)) \right \rangle.
\]
Differentiating in $s$ yields
\begin{equation} \label{2t}
    \partial_s^2 \phi_\alpha(s_0,t_0) = \left \langle \frac{D}{ds} \tgamma_\alpha'(s_0), \partial_r \zeta(s_0,r_0) \right \rangle + \left \langle \tgamma_\alpha'(s_0), \left. \frac{D}{ds} \right|_{s = s_0} \partial_r \zeta(s,\phi(s,t_0)) \right \rangle.
\end{equation}
Now
\begin{align*}
    \left. \frac{D}{ds} \right|_{s = s_0} \partial_r \zeta(s,\phi(s,t_0)) &= \frac{D}{ds} \partial_r \zeta(s_0, r_0) + \partial_s \phi_\alpha(s_0,t_0) \frac{D}{dr} \partial_r \zeta (s_0, r_0).
\end{align*}
The second term on the right vanishes since $r \mapsto \zeta(s_0,r)$ is a geodesic. The curve $s \mapsto \zeta(s,r_0)$ is a geodesic circle of radius $r_0$. Hence, $\partial_r \zeta$ and $\partial_s \zeta$ are perpendicular by Gauss' lemma. Since in addition $|\partial_r \zeta| = 1$, there exists a scalar $\kappa$ such that
\begin{equation} \label{def kappa}
    \frac{D}{ds} \partial_r \zeta = \kappa \partial_s \zeta.
\end{equation}
In fact, $\kappa(s_0,r_0)$ is the geodesic curvature of the circle $s \mapsto \zeta(s,r_0)$ at $s = s_0$.
Hence, we write \eqref{2t} as
\begin{equation} \label{2t 2}
    \partial_s^2 \phi_\alpha(s_0,t_0) = \left \langle \frac{D}{ds} \tgamma_\alpha'(s_0), \partial_r \zeta(s_0, r_0) \right \rangle + \kappa(s_0,r_0) \left \langle \tgamma_\alpha'(s_0), \partial_s \zeta(s_0,r_0) \right \rangle.
\end{equation}
Let $\theta \in [0, \pi/2]$ denote the angle of intersection between the curve $\tgamma_\alpha$ and the circle $s \mapsto \zeta(s,r_0)$ (see Figure \ref{theta fig}).
\begin{figure}
	\centering
	\includegraphics[width=0.8\textwidth]{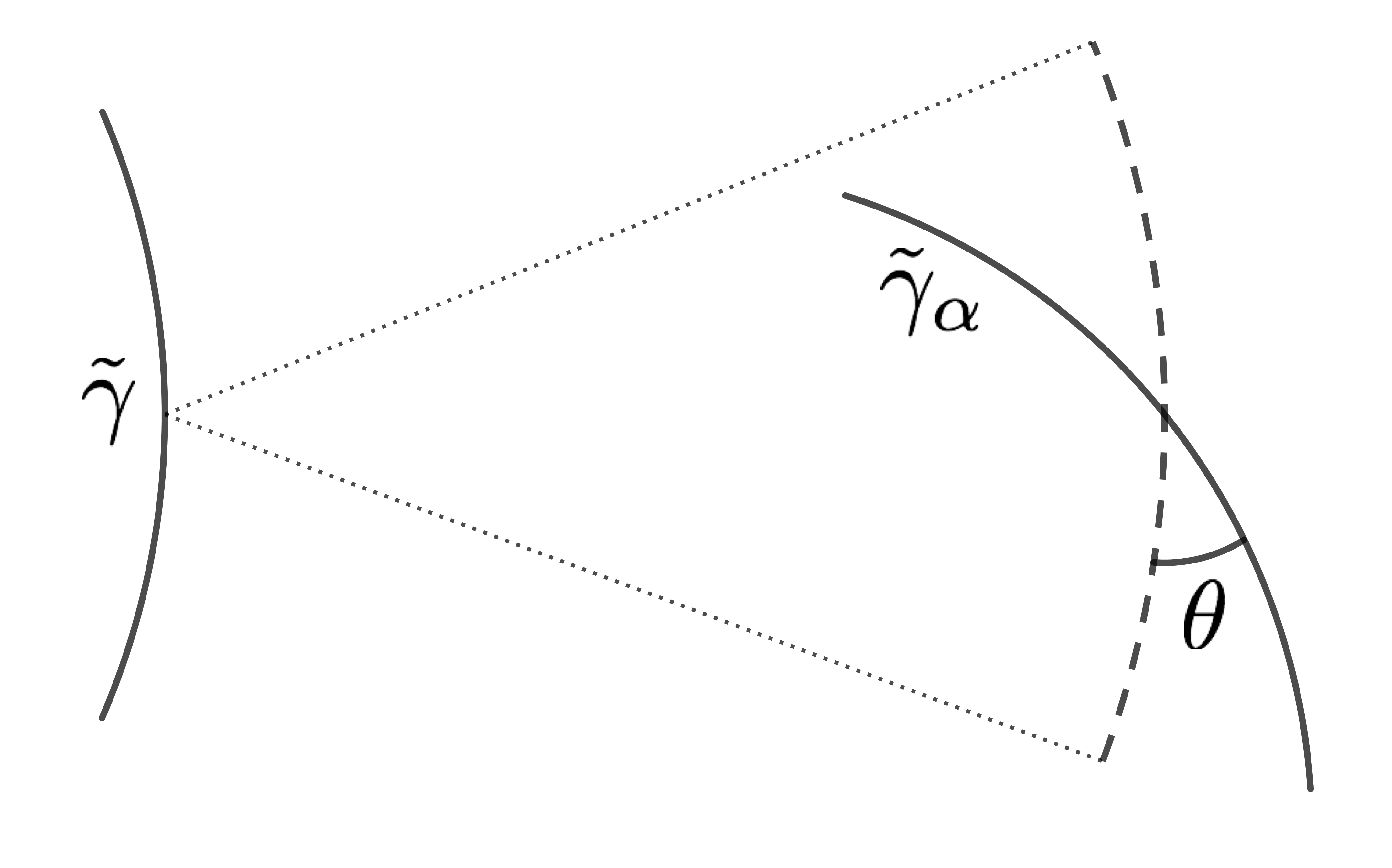}
	\caption{$\tilde \gamma$ and $\tilde \gamma_\alpha$ (solid curves, on the left and right, resp.) and the angle $\theta$ between $\tilde \gamma_\alpha$ and a geodesic circle with center on $\tilde \gamma$.}
	\label{theta fig}
\end{figure}
We have
\[
    \tgamma_\alpha'(s_0) = \left. \frac{\partial}{\partial s} \right|_{s = s_0} \zeta(s,\phi(s,t_0)) = \partial_s \zeta(s_0,r_0) + \partial_s \phi(s_0,t_0) \partial_r \zeta(s_0,r_0),
\]
and since $\partial_r \zeta$ and $\partial_s \zeta$ are perpendicular,
\[
    \langle \tgamma_\alpha'(s_0) , \partial_s \zeta(s_0,r_0) \rangle = |\partial_s \zeta(s_0,r_0)|^2 = \cos^2(\theta).
\]
The line above and \eqref{2t 2} yield the key computation
\begin{equation} \label{pure second}
    \partial_s^2 \phi_\alpha(s_0,t_0) = \cos(\theta) ( \pm \kappa_{\gamma}(s_0) + \cos(\theta) \kappa(s_0,r_0) ),
\end{equation}
where $\pm$ matches the sign of $\langle \frac{D}{ds} \tgamma_\alpha', \partial_r \zeta \rangle$. This computation is the last thing we need to prove Proposition \ref{medium time}.


\begin{proof}[Proof of Proposition \ref{medium time}]
    Since $A$ is fixed and finite, we may restrict the support of $b$ without worrying about doing so uniformly over elements of $A$. Fix $\alpha \in A \setminus I$. Let $D$ denote the diagonal of $\mathcal I \times \mathcal I$. We claim that that
    \begin{equation} \label{pf medium time 1}
        D \subset \{ \partial_t^2 \phi_\alpha \neq 0\} \cup \{ \partial_s^2 \phi_\alpha \neq 0\} \cup \{ \nabla \phi_\alpha \neq 0 \}
    \end{equation}
    Provided our claim holds, we take the support of $b$ small enough so that $\supp b \times \supp b$ lies entirely in one of the three open sets above. If $\supp b \times \supp b$ lies in one of the first two sets, the proposition follows by stationary phase ~\cite[Theorem 1.1.1]{SFIO} in the appropriate variable. If $\supp b \times \supp b$ is contained in the third set, then the proposition follows from nonstationary phase ~\cite[Lemma 0.4.7]{SFIO}.
    
    To prove \eqref{pf medium time 1}, we suppose $s_0 = t_0$ and $\nabla \phi_\alpha(s_0,t_0) = 0$ and $\partial_s^2 \phi_\alpha(s_0,t_0) = 0$ and show that $\partial_t^2 \phi_\alpha(s_0,t_0) \neq 0$. By \eqref{pure second},
    \[
        0 = \pm \kappa_\gamma(s_0) + \kappa(s_0,r_0)
    \]
    where $r_0 = \phi_\alpha(s_0,t_0)$. This tells us $\pm$ must take the negative sign, and that $\tgamma_\alpha$ is curved ``towards" $\tgamma(t_0)$. This means that $\tgamma$ must be curved ``away" from $\tgamma_\alpha(s_0)$, since otherwise the midpoint of the geodesic connecting $\tgamma(t_0)$ and $\tgamma_\alpha(s_0)$ would be fixed by the deck transformation $\alpha$ (see Figure \ref{periodicfig}). Hence $\partial_t^2 \phi_\alpha(s_0,t_0) > 0$ by the analogous computation to \eqref{pure second}.
\begin{figure}
	\centering
	\includegraphics[width=0.9\textwidth]{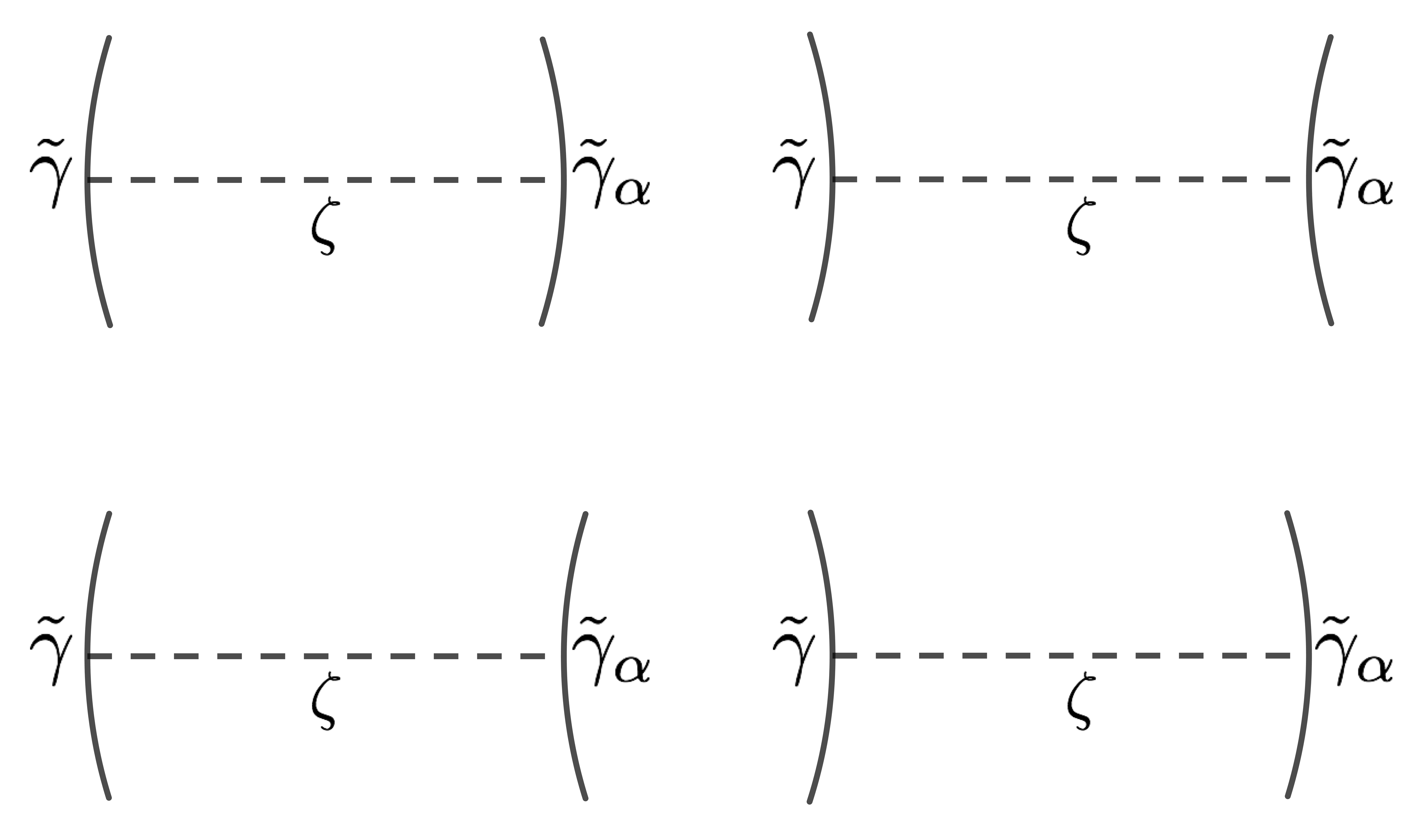}
	\caption{Depictions of $\tilde \gamma$ and $\tilde \gamma_\alpha$ connected by a normal geodesic $\zeta$ (the dashed line) adjoining $\tilde \gamma(t_0)$ and $\tilde \gamma_\alpha(s_0)$ at $s_0 = t_0$. The top two situations are impossible, otherwise the midpoint of $\zeta$ would be a fixed point of $\alpha$.}
	\label{periodicfig}
\end{figure}
\end{proof}


\section{Curvature of circles}\label{CURVATURE OF CIRCLES}

Fix $t_0$ and let $\kappa$ be as in \eqref{def kappa}. To apply \eqref{pure second} in any useful way, we need to know something about the function $\kappa(s,r)$, the curvature of a geodesic circle of radius $r$ centered $\tgamma(t_0)$, and how it relates to $\bfk$. Before this, though, we verify Definition \ref{def k}.

\begin{proposition} \label{def k good}
	The function $\bfk : SM \to \R$ in Definition \ref{def k} is well-defined, continuous, and nonnegative.
\end{proposition}

\begin{proof}
	Let $\zeta$ and $J$ be as in Definition \ref{def k}. By properties of Jacobi fields, $\bfk(v)$ is independent of the choice of direction of $J(0)$, and so we fix one. Since both $J(0)$ and $\frac{D}{dr} J(0)$ are perpendicular to $\zeta'(0)$,
    \[
        \langle J(r), \zeta'(r) \rangle = 0 \quad \text{ for all } r.
    \]
    We write
    \[
        J(r) = h(r) w(r)
    \]
    where $w(r)$ is the vector at $\zeta(r)$ obtained through a parallel transport of $J(0)$ and where $h$ satisfies
    \begin{align}
        \label{h jacobi equation} h'' + Kh &= 0 \quad \text{ and }\\
        \label{h initial condition 0} h(0) &= 1.
    \end{align}
    Here $K = K(\zeta(r))$ is the sectional curvature at $\zeta(r)$. It suffices to show the existence and uniqueness of such an $h$ with
    \begin{equation} \label{h bounded}
        |h(r)| \leq C \quad \text{ for } r \leq 0
    \end{equation}
    and afterwards set $\mathbf k(v) = h'(0)$. Note $\bfk(v)$ is necessarily nonnegative since otherwise \eqref{h bounded} would break by convexity. The difference of two such functions $h_1$ and $h_2$ satisfies \eqref{h jacobi equation} and \eqref{h bounded} and $(h_1 - h_2)'(0) = 0$. If $(h_1 - h_2)'(0) \neq 0$, then $(h_1 - h_2)(r)$ is unbounded for $r \leq 0$ by convexity. Hence, $h_1 = h_2$ identically. This proves uniqueness.

    To prove existence, we construct a bounded $h$ as a limit. For all $s > 0$, let $h_{s}$ denote the unique function satisfying \eqref{h jacobi equation}, \eqref{h initial condition 0}, and $h_s(-s) = 0$. We construct as a limit
    \begin{equation} \label{h limit}
        h_{\infty} = \lim_{s \to \infty} h_s = h_{1} + \int_1^\infty \partial_s h_s \, ds.
    \end{equation}
    We will show
    \begin{equation} \label{h t derivative bound}
        |\partial_s h_s(r)| \leq -\frac{r}{s^2} \quad \text{ for } r \leq 0,
    \end{equation}
    which guarantees uniform convergence of \eqref{h limit} for $r$ in compact sets, whence $h_{\infty}$ satisfies \eqref{h jacobi equation} by regularity. Moreover for $r \leq 0$,
    \[
    	h_\infty(r) = \int_{-r}^\infty \partial_s h_{s}(r) \, ds \leq 1
    \]
    which is stronger than required. We now prove \eqref{h t derivative bound}. $h_s$ vanishes uniquely at $-s$ since $K \leq 0$. Hence, $h_s \geq 0$ and so $h_s'' \geq 0$ on $[-s,0]$. Since $h_s(0) = 1$ and $h_s(-s) = 0$,
    \[
        0 \leq h_s(r) \leq \left( 1 + \frac{r}{s} \right) \quad \text{ for } -s \leq r \leq 0
    \]
    by convexity. We conclude that
    \[
        0 < h_s'(-s) \leq \frac{1}{s}
    \]
    by writing $h_s'(-s)$ using the limit definition of the derivative and applying the previous inequality. Since $h_s(-s) = 0$ for all $s > 0$,
    \[
        0 = \frac{d}{ds} h_s(-s) = -h_s'(-s) + \partial_s h_s(-s)
    \]
    whence
    \begin{equation} \label{dt ht endpoint bound}
        0 < \partial_s h_s(-s) \leq \frac{1}{s}.
    \end{equation}
    $\partial_s h_s$ satisfies \eqref{h jacobi equation} with initial data $\partial_s h_s(0) = 0$. Since $\partial_s h_s(-s) > 0$, a similar convexity argument yields
    \[
        0 < \partial_s h_s(r) \leq -\partial_s h_s(-s) \frac{r}{s} \quad \text{ for } -s \leq r < 0.
    \]
    \eqref{h t derivative bound} follows from the above inequality and \eqref{dt ht endpoint bound}.
    
    Finally, we show $\mathbf k$ is continuous on $SM$. To do so, we show that $\mathbf k$ is continuous on every continuous path $t \mapsto v(t)$ in $SM$. If $r \mapsto \zeta(t,r)$ is the geodesic with $\partial_r \zeta(t,0) = v(t)$, we let $h_\infty(t,r)$ and $h_s(t,r)$ be as constructed above along the geodesic $r \mapsto \zeta(t,r)$. In the limit as $t \to 0$, the sectional curvature $K(\zeta(t,r))$ converges to $K(\zeta(0,r))$ uniformly for $r$ in a compact set. Combined with \eqref{h jacobi equation}, we have for any $\epsilon > 0$ and $s > 0$ a $\delta > 0$ such that
    \[
        |h_s(t,r) - h_s(0,r)| < \frac{\epsilon}{3}  \quad \text{ for } -s \leq r \leq 0
    \]
    if $|t| < \delta$. Moreover if $r$ lies in some compact set, by \eqref{h t derivative bound} and the fundamental theorem of calculus, there exists $s > 0$ large enough such that
    \[
        |h_\infty(t,r) - h_s(t,r) | < \frac{\epsilon}{3}
    \]
    independently of $t$. Putting these bounds together, we have
    \begin{align*}
        |h_\infty(t,r) - &h_\infty(0,r)|\\
        &\leq |h_\infty(t,r) - h_s(t,r)|  + |h_s(t,r) - h_s(0,r)|  + |h_s(0,r) - h_\infty(0,r)| < \epsilon,
    \end{align*}
    i.e. $h_\infty(t,r) \to h_\infty(0,r)$ uniformly for $r$ in a compact set. By regularity, $\partial_r h(t,r) \to \partial_r h(0,r)$ as $t \to 0$, and in particular $\mathbf k(v(t)) \to \mathbf k(v(0))$.
\end{proof}

We lift $\bfk$ to the universal cover. Let $\tilde \bfk$ denote the function on $S\R^2$ for which $\tilde \bfk(\tilde v) = \bfk(v)$ wherever $\tilde v \in S\R^2$ is a lift of $v \in SM$. Since the covering map is a local isometry, $\tilde \bfk$ satisfies Definition \ref{def k} where $M$ in the definition is replaced with the universal cover. Furthermore, we can loosen Definition \ref{def k} a little bit. Let $v$, $\zeta$, and $J$ be as in Definition \ref{def k} except with $M$ replaced by the universal cover. If $|J(0)|$ is allowed to be any positive number, we have
\[
    \tilde \bfk(v) = \frac{|\frac{D}{dr} J(0)|}{|J(0)|}.
\]
Since $J$ is a perpendicular Jacobi field along $\zeta$,
\[
    J(r) = h(r) w(r)
\]
where $w$ is a unit normal vector field along $\zeta$ and $h$ satisfies
\[
    h''(r) + \tilde K(\zeta(r))h(r) = 0.
\]
We conclude that
\[
    \tilde \bfk(\zeta'(r)) = \frac{|\frac{D}{dr} J(r)|}{|J(r)|} = \frac{h'(r)}{h(r)} \qquad \text{ for all } r \geq 0
\]
It follows
\begin{equation} \label{k ode}
    \frac{d}{dr} \tilde \bfk(\zeta'(r)) = \frac{h''(r)}{h(r)} - \frac{h'(r)^2}{h(r)^2} = - \tilde K(\zeta(r)) - \tilde \bfk(\zeta'(r))^2.
\end{equation}

Note by the same argument for \eqref{st commute},
\[
    \frac{D}{ds} \partial_r \zeta = \frac{D}{dr} \partial_s \zeta.
\]
This and \eqref{def kappa} yields
\[
    \frac{D^2}{dr^2} \partial_s \zeta = \frac{D}{dr} \left( \kappa \partial_s \zeta \right) = \partial_r \kappa \partial_s \zeta + \kappa \frac{D}{dr} \partial_s \zeta =  (\partial_r \kappa + \kappa^2) \partial_s \zeta.
\]
On the other hand since $\partial_s \zeta$ is a perpendicular Jacobi field along $r \mapsto \zeta(s,r)$,
\[
    \frac{D^2}{dr^2} \partial_s \zeta = -\tilde K \partial_s \zeta.
\]
Putting these together, we obtain
\begin{equation} \label{kappa ode}
    \partial_r \kappa + \kappa^2 + \tilde K = 0,
\end{equation}
the same differential equation that $\bfk$ satisfies in \eqref{k ode}. As a consequence, we deduce the following facts.

\begin{lemma} \label{large radius}
    Let $r \mapsto \zeta(r)$ be a unit-speed geodesic in $(\R^2,\tilde g)$ and $\kappa(r)$ the geodesic curvature at $\zeta(r)$ of the circle of radius $r$ with center at $\zeta(0)$. Moreover, suppose $0 \geq K_0 \geq K \geq K_1$ for some constants $K_0$ and $K_1$. The following are true.
    \begin{enumerate}
    \item $\displaystyle 0 < \kappa(r) - \tilde \bfk(\partial_r \zeta(r)) \leq r^{-1}$ for all $r > 0$.
    \item $\displaystyle \sqrt{-K_0} \leq \tilde \bfk(v) \leq \sqrt{-K_1}$ for all $v \in S\tilde M$.
    \item For $r > 0$,
    \[
    	\kappa(r) \geq \begin{cases} \sqrt{-K_0} \coth(\sqrt{-K_0} r) & \text{ if } K_0 < 0, \\ 1/r & \text{ if } K_0 = 0. \end{cases}
	\]
    \end{enumerate}
\end{lemma}

\begin{proof}
    (1) Since both $\kappa$ and $\tilde \bfk$ satisfy \eqref{kappa ode}, the difference $\kappa - \tilde \bfk$ satisfies
    \begin{equation} \label{difference ode}
        \partial_r (\kappa - \tilde \bfk) = -(\kappa^2 - \tilde \bfk^2).
    \end{equation}
    Since $\kappa(r)$ is large for small $r$, we can easily guarantee that $\kappa(r_0) > \tilde \bfk(\zeta'(r))$ for $0 < r \ll 1$. Now $\kappa$ and $\tilde \bfk$ are smooth for $r > 0$, and since $\kappa - \tilde \bfk = 0$ is an equilibrium of \eqref{difference ode}, we have that
    \[
        \kappa(r) - \tilde \bfk(\zeta'(r)) > 0 \qquad \text{ for all } r > 0.
    \]
    Hence we rephrase \eqref{difference ode} and obtain
    \[
        \partial_r (\kappa - \tilde \bfk) = -\frac{\kappa + \tilde \bfk}{\kappa - \tilde \bfk}(\kappa - \tilde \bfk)^2 \leq -(\kappa - \tilde \bfk)^2,
    \]
    the inequality a consequence of the fact that $\kappa > \tilde \bfk$. We then have by comparison
    \[
        \kappa(r) - \tilde \bfk(\zeta'(r)) \leq r^{-1},
    \]
    as desired.
    
    (2) $\tilde \bfk$ is nonnegative and bounded by Proposition \ref{def k good} and the compactness of $M$. Suppose $\tilde \bfk(\zeta'(0)) < \sqrt{-K_0}$. By \eqref{k ode}
    \[
    	\partial_r \tilde \bfk(\zeta'(r)) > 0 \qquad \text{ wherever } \qquad \tilde \bfk(\zeta'(r)) < \sqrt{-K_0},
	\]
	and hence
	\[
		\tilde \bfk(\zeta'(r)) < \sqrt{-K_0} \qquad \text{ for all } r < 0.
	\]
	We conclude that $\partial_r \tilde \bfk(\zeta'(r))$ is positive bounded away from $0$ for $r < 0$, and hence $\tilde \bfk(\zeta'(r))$ is eventually negative if $r$ is negative enough. Hence, $\tilde \bfk(\zeta'(0)) \geq \sqrt{-K_0}$. A similar argument shows $\tilde \bfk(\zeta'(r))$ is unbounded for $r < 0$ if $\tilde \bfk(\zeta'(0)) > \sqrt{-K_1}$.
    
    (3) Geometric considerations show $\kappa$ has initial data
    \begin{equation} \label{geometric initial data}
    	\lim_{r \searrow 0} r \kappa(r) = 1.
    \end{equation}
    Now,
    \[
    	\kappa'(r) \geq -K_0 - \kappa(r)^2 \qquad \text{ for } r > 0,
    \]
    and part (3) follows from comparison with the ordinary differential equation $u' = -K_0 - u^2$ with the initial data \eqref{geometric initial data} and an elementary computation.
\end{proof}

We now use the computation \eqref{pure second} of $\partial_s^2 \phi_\alpha$ and Lemma \ref{large radius} to prove some uniform bounds on the derivatives of $\phi_\alpha$ to be used in the proof of Proposition \ref{large time}. Recall $\mathcal I$ is some open interval containing the support of $b$ on which $\gamma$ satisfies the hypotheses \eqref{curvature hypotheses} of Theorem \ref{main theorem}.

\begin{lemma} \label{pure second bounds}
    Suppose
    \[
        |\kappa_{\gamma}(s) - \kappa(s, \phi_\alpha(s,t))| > \epsilon \qquad \text{ for all } s,t \in \mathcal I
    \]
    for some $\epsilon > 0$. Then there exist positive constants $\delta$ and $\eta$ independent of $\alpha$ such that if the diameter of $\mathcal I$ is less than $\delta$ and $\partial_s \phi_\alpha$ is nonvanishing on $\mathcal I \times \mathcal I$, then
    \[
        |\partial_s \phi_\alpha(s,t)| \geq \eta \qquad \text{ for } s,t \in \supp b.
    \]
    On the other hand if $\partial_s \phi_\alpha(s_0,t_0) = 0$ for some $s_0,t_0 \in \mathcal I$, then
    \[
        |\partial_s^2 \phi_\alpha(s,t)| \geq \epsilon/4 \quad \text{ for } s,t \in \mathcal I.
    \]
    This result holds similarly for derivatives in $t$.
\end{lemma}

\begin{proof}
    The curvature of any geodesic circle in $(\R^2,\tilde g)$ with radius at least $1$ is bounded uniformly by Lemma \ref{large radius} and the fact that $\tilde \bfk$ is bounded. Hence, we select a global constant $C$ so that
    \[
        \sup_{s \in \mathcal I} \kappa_\gamma(s) + \sup_{s,t \in \mathcal I} \kappa(s,\phi_\alpha(s,t)) \leq C \qquad \text{ for all } \alpha \neq I
    \]
    where $\kappa(s,\phi_\alpha(s,t))$ is the curvature of the circle at $\tgamma_\alpha(s)$, with center at $\tgamma(t)$ and radius $\phi_\alpha(s,t)$ as per \eqref{def kappa}.
    Set
    \[
        \eta' = \min\left( \frac{1}{2}, \frac{\epsilon}{2C} \right).
    \]
    Our first claim is that
    \[
        |\partial_s^2 \phi_\alpha| \geq \epsilon/4 \quad \text{ if } \quad |\partial_s \phi_\alpha| \leq \eta'.
    \]
    Note first
    \[
        \sin(\theta) = |\partial_s \phi_\alpha(s,t_0)|
    \]
    (see Figure \ref{theta fig}), then if $|\partial_s \phi_\alpha(s,t_0)| \leq \eta'$,
    \[
        \cos(\theta) \geq \sqrt{1 - \eta'^2} \geq 1 - \epsilon/2C.
    \]
    Since $\eta' \leq 1/2$, we have that $\cos(\theta) \geq 1/2$ by default.
    Hence,
    \begin{align*}
        |\partial_s^2 \phi_\alpha| &\geq \frac{1}{2}\left| \pm \kappa_{\gamma} + \cos(\theta) \kappa \right|\\
                            &= \frac{1}{2}\left| \pm \kappa_{\gamma} + \kappa - (1 - \cos(\theta)) \kappa \right|\\
                            &\geq \frac{1}{2} | \kappa_{\gamma} - \kappa | - \frac{1}{2}|1 - \cos(\theta)| |\kappa|\\
                            &\geq \frac{1}{2} \epsilon -\frac{1}{2} \frac{C\epsilon}{2C}\\
                            &= \frac{\epsilon}{4},
    \end{align*}
    proving our claim.
    
    Set
    \[
        \delta = \frac{\eta'}{2\sqrt{2}(1+C^2)^{1/2}}.
    \]
    By \eqref{pure second},
    \[
        |\partial_s^2 \phi_\alpha| \leq C.
    \]
    Moreover by Lemma \ref{mixed bound}, the fact that $\mathcal I$ has diameter at most $1$, and that the injectivity radius is at least $10$, we have
    \[
        |\partial_t \partial_s \phi_\alpha(s,t)| \leq 1.
    \]
    Hence for any $(s,t)$ and $(s_0,t_0)$ in $\I \times \I$,
    \begin{align*}
        |\partial_s \phi_\alpha(s,t) - \partial_s \phi_\alpha(s_0,t_0)| &\leq (1 + C^2)^{1/2}|(s,t) - (s_0,t_0)| \leq \frac{\eta'}{2}
    \end{align*}
    since the diameter of $\I \times \I$ is no greater than $\sqrt 2 \delta$.
    In particular if $\partial_s \phi_\alpha(s_0,t_0) = 0$, then
    \[
        |\partial_s \phi_\alpha(s,t)| \leq \eta'/2 \qquad \text{ for all } s,t \in \mathcal I
    \]
    and so $|\partial_s^2 \phi_\alpha(s,t)| \geq \epsilon/4$ by our claim.
    
    Now suppose $|\partial_s \phi_\alpha(s,t)| > 0$ for all $s,t \in \mathcal I$. In the case that $|\partial_s \phi_\alpha(s_0,t_0)| \leq \eta'/2$ for some $s_0,t_0 \in \mathcal I$, $|\partial_s \phi_\alpha(s,t)| \leq \eta'$ for all $s,t \in \mathcal I$, and hence $\partial_s \phi_\alpha(s,t)$ is monotonic in $s$, and so $\partial_s \phi_\alpha$ is smallest near an endpoint of $\mathcal I$. Since $\supp b$ is closed and $\mathcal I$ open, the distance $d(\supp b, \mathcal I^c)$ from $\supp b$ to the complement of $\mathcal I$ is positive. Hence,
    \[
        |\partial_s \phi(s,t)| \geq d(\supp b, \mathcal I^c) \epsilon/4 > 0.
    \]
    The proof is complete after setting
    \[
        \eta = \min(\eta'/2, d(\supp b, \mathcal I^c) \epsilon/4).
    \]
\end{proof}


\section{Proof of Proposition \ref{large time}}

    By our hypotheses \eqref{curvature hypotheses} on the curvature of $\gamma$, and since $\bfk$ is continuous, we restrict the support of $b$ and also the interval $\mathcal I$ so that
    \[
        \inf_{s,t \in \mathcal I} |\kappa_\gamma(t) - \bfk(\pm \gamma'^\perp(s))| > 2\epsilon
    \]
    for some small $\epsilon > 0$. We first require $R$ in \eqref{def A} be at least as large as $16\epsilon^{-1}$ so that, by Lemma \ref{mixed bound},
    \begin{equation} \label{phase function mixed}
        \sup_{t,s \in \mathcal I} |\partial_t \partial_s \phi_\alpha(s,t)| \leq \epsilon/8 \qquad \text{ if } \alpha \in \Gamma \setminus A.
    \end{equation}
    Let $\zeta$ be defined as in Section \ref{PHASE FUNCTION BOUNDS}, that is let $r \mapsto \zeta(s,t,r)$ be the unit-speed geodesic with $\zeta(s,t,0) = \tgamma(t)$ and $\zeta(s,t,\phi_\alpha(s,t)) = \tgamma_\alpha(s)$. Moreover let $\kappa(s,t)$ denote the curvature at $\tgamma_\alpha(s)$ of the circle with center $\tgamma(t)$ and radius $\phi_\alpha(s,t)$ (see Figure \ref{theta fig}). By Lemma \ref{large radius} and our requirement that $R > 16 \epsilon^{-1}$,
    \[
        |\kappa(s,t) - \tilde \bfk(\partial_r \zeta(s,t,\phi_\alpha(s,t)))| < \epsilon.
    \]
    Hence,
    \begin{equation} \label{i'm tired}
        |\kappa(s,t) - \kappa_{\gamma_\alpha}(t)| > \epsilon \qquad \text{ for } s,t \in \mathcal I.
    \end{equation}
    
    To summarize, we need to show
    \begin{equation}\label{final integral}
       \left| \iint a(s,t) e^{\pm i \lambda \phi_\alpha(s,t)} \, ds \, dt \right| \leq C e^{CT} \lambda^{-1}
    \end{equation}
    where $C$ is independent of $\alpha$, and where for simplicity we have written
    \[
        a(s,t) = b(s,t) a_\pm(T,\lambda; \tgamma_\alpha(s), \tgamma(t)).
    \]
    Considering \eqref{i'm tired}, we restrict the diameter of $\mathcal I$ to be less than the $\delta$ in Lemma \ref{pure second bounds}. If $|\partial_s \phi_\alpha| > 0$ on $\mathcal I \times \mathcal I$, $|\partial_s \phi_\alpha| \geq \eta$ on $\supp b \times \supp b$ for some $\eta > 0$ independent of $\alpha$. Then we integrate by parts to write the integral in \eqref{final integral} as
    \[
        \frac{1}{\pm i\lambda} \iint \partial_s \left( \frac{a(s,t)}{\partial_s \phi_\alpha(s,t)} \right) e^{\pm i \lambda \phi_\alpha(s,t)} \, ds \, dt.
    \]
    Now
    \begin{equation}\label{i'm very very tired}
        |\partial_s a(s,t)| \lesssim e^{CT}
    \end{equation}
    by \eqref{KB a bounded} and \eqref{KB a bound x}, and
    since $|\partial_s \phi_\alpha| > \eta$, we have
    \[
        \partial_s \left( \frac{a}{\partial_t \phi_\alpha} \right) \leq Ce^{CT},
    \]
    from which the desired bound follows. We obtain the desired bound similarly if $\partial_t \phi_\alpha$ does not vanish in $\mathcal I \times \mathcal I$.
    
    By Lemma \ref{pure second bounds}, all that is left is the case that $\nabla \phi_\alpha$ vanishes at exactly one point $(s_0,t_0) \in \mathcal I \times \mathcal I$. By a translation, we assume without loss of generality that $(s_0,t_0) = (0,0)$. In this case,
    \begin{equation} \label{diagonal bounds}
        |\partial_s^2 \phi_\alpha| \geq \epsilon/4 \quad \text{ and } \quad |\partial_t^2 \phi_\alpha| \geq \epsilon/4
    \end{equation}
    on $\mathcal I \times \mathcal I$. We use a careful stationary phase argument to obtain \eqref{final integral}. By \eqref{phase function mixed} and \eqref{diagonal bounds},
    \[
    	\left| \begin{bmatrix}
			\partial_s^2 \phi_\alpha & \partial_s \partial_t \phi_\alpha \\
			\partial_s \partial_t \phi_\alpha & \partial_t^2 \phi_\alpha
		\end{bmatrix} \xi \right| \geq \frac{\epsilon}{8} |\xi| \qquad \text{ for all } \xi \in \R^2.
    \]
    Hence by the mean value theorem, there exists $c'$ depending only on $\epsilon$ such that
    \begin{equation} \label{gradient phi bounds}
        |\nabla \phi(s,t)| \geq c' |(s,t)| \qquad \text{ for all } s,t \in \mathcal I.
    \end{equation}
    Let $\psi \in C_0^\infty(\R^2)$ with $\psi(s,t) = 1$ for $|(s,t)| \leq 1/2$ and $\psi(s,t) = 0$ for $|(s,t)| \geq 1$. We write the integral in \eqref{final integral} as the sum of respective parts
    \begin{align*}
        I + II = \iint &\psi(\lambda^{1/2} s, \lambda^{1/2} t ) a(s,t) e^{\pm i \lambda \phi_\alpha(s,t)} \, ds \, dt\\
        &+ \iint (1 - \psi(\lambda^{1/2} s, \lambda^{1/2} t )) a(s,t) e^{\pm i \lambda \phi_\alpha(s,t)} \, ds \, dt.
    \end{align*}
    We have trivially
    \[
        |I| \leq Ce^{CT} \lambda^{-1},
    \]
    and so it suffices to bound $II$. We define an operator
    \[
        L = \frac{\nabla \phi}{\pm i \lambda |\nabla \phi|^2} \cdot \nabla
    \]
    with adjoint
    \[
        L^* f = -\frac{1}{\pm i\lambda} \left[ \partial_s \left( \frac{\partial_s \phi_\alpha}{|\nabla \phi_\alpha|^2} f \right) + \partial_t \left( \frac{\partial_t \phi_\alpha}{|\nabla \phi_\alpha|^2} f \right) \right].
    \]
    Then since
    \[
        L e^{\pm i \lambda \phi_\alpha} = e^{\pm i \lambda \phi_\alpha},
    \]
    we write
    \begin{align*}
        II &= \iint (1 - \psi(\lambda^{1/2}s, \lambda^{1/2}t)) a(s,t) Le^{\pm i \lambda \phi_\alpha(s,t)} \, ds \, dt \\
        &= \iint L^*[(1 - \psi(\lambda^{1/2}s, \lambda^{1/2}t)) a(s,t)] e^{\pm i \lambda \phi_\alpha(s,t)} \, ds \, dt.
    \end{align*}
    Firstly,
    \[
        \left| \partial_s \frac{\partial_s \phi_\alpha}{|\nabla \phi_\alpha|^2} \right| \leq C |\nabla \phi_\alpha|^{-2} \leq C |(s,t)|^{-2}.
    \]
    since the first and second derivatives of $\phi_\alpha$ are bounded by a constant uniform for $\alpha \neq I$, as established by Lemma \ref{mixed bound} and the proof of Lemma \ref{pure second bounds}. Secondly,
    \[
        \left| \partial_s (1 - \psi(\lambda^{1/2}s, \lambda^{1/2}t)) \right| \leq \lambda^{1/2} |\partial_s \psi(\lambda^{1/2}s,\lambda^{1/2}t)| \leq C|(s,t)|^{-1}.
    \]
    All bounds hold similarly for the derivative in $t$. Hence by \eqref{i'm very very tired},
    \[
        |L^*[(1 - \psi(\lambda^{1/2}s, \lambda^{1/2}t)) a(s,t)]| \lesssim e^{CT} \lambda^{-1}|(s,t)|^{-2}
    \]
    and so using polar coordinates,
    \[
        |II| \lesssim e^{CT} \lambda^{-1} \int_{\lambda^{-1/2}}^{\diam \I \times \I} r^{-2} r \, dr \lesssim e^{CT} \lambda^{-1} \log \lambda.
    \]
    By \eqref{T = clog}, we absorb $\log \lambda$ into $e^{CT}$ and obtain the desired bound.

\bibliography{references}{}
\bibliographystyle{alpha}

\end{document}